\documentclass[a4paper,11pt]{amsart}
\usepackage{amssymb,amsmath}
\usepackage{graphicx}
\usepackage{hyperref}
\usepackage[utf8]{inputenc}
\usepackage{cmtt}

\newcommand{\TODO}[1]{}

\newtheorem{theorem}{Theorem}

\theoremstyle{remark}

\let\leq\leqslant
\let\geq\geqslant

\let\setminus\smallsetminus
\let\epsilon\varepsilon
\let\epsi\varepsilon

\def\calG{\mathcal{G}}
\def\calC{\mathcal{C}}

\DeclareMathOperator{\lev}{lev}
\let\ts\undefined
\DeclareMathOperator{\ts}{ts}
\DeclareMathOperator{\sel}{sel}

\newcommand{\brac}[1]{{\left(#1\right)}}
\newcommand{\set}[1]{\left\{#1\right\}}
\newcommand{\norm}[1]{{\left|#1\right|}}
\newcommand{\floor}[1]{{\left\lfloor #1 \right\rfloor}}
\newcommand{\ceil}[1]{{\left\lceil #1 \right\rceil}}

\let\O\undefined
\newcommand{\O}[1]{O\left(#1\right)}

\newcommand{\Om}[1]{\Omega\left(#1\right)}

\newcommand{\Rule}[1]{\textbf{Rule #1:}}

\makeatletter
\let\old@setaddresses\@setaddresses
\def\@setaddresses{\bgroup\parindent 0pt\let\scshape\relax\old@setaddresses\egroup}
\makeatother

\begin{document}

\title{Lower Bounds for On-line Graph Colorings}

\author[G.~Gutowski]{Grzegorz Gutowski}
\author[J.~Kozik]{Jakub Kozik}
\author[P.~Micek]{Piotr Micek}
\author[X.~Zhu]{Xuding Zhu}

\address[G.~Gutowski, J.~Kozik, P.~Micek]{Theoretical Computer Science Department, Faculty of Mathematics and Computer Science, Jagiellonian University, Krak\'{o}w, Poland}
\email{\mtt\{gutowski,jkozik,micek\mtt\}@tcs.uj.edu.pl}

\address[X.~Zhu]{Department of Mathematics, Zhejiang Normal University, China, and Department of Applied Mathematics, National Sun Yat-sen University, Taiwan}
\email{xdzhu@zjnu.edu.cn}

\thanks{This research is supported by: Polish National Science Center UMO-2011/03/D/ST6/01370.}

\begin{abstract}
We propose two strategies for Presenter in on-line graph coloring games.
The first one constructs bipartite graphs and forces any on-line coloring algorithm to use $2\log_2 n - 10$ colors, where $n$ is the number of vertices in the constructed graph.
This is best possible up to an additive constant.
The second strategy constructs graphs that contain neither $C_3$ nor $C_5$ as a subgraph and forces $\Omega(\frac{n}{\log n}^\frac{1}{3})$ colors.
The best known on-line coloring algorithm for these graphs uses $O(n^{\frac{1}{2}})$ colors.
\end{abstract}

\maketitle

\section{Introduction}

A \emph{proper coloring} of a graph $G$ is an assignment of colors to the vertices of the graph such that adjacent vertices receive distinct colors.
An $n$-round \emph{on-line coloring game} on a class of graphs $\calG$ is a two-person game, played by Presenter and Algorithm.
In each round Presenter introduces a new vertex of a graph with its adjacency status to all vertices presented earlier.
The only restriction for Presenter is that in every moment the currently presented graph is in $\calG$.
Algorithm assigns colors to the incoming vertices in such a way that the coloring of the presented graph is always proper.
The color for a new vertex has to be assigned before Presenter introduces the next vertex.
The assignment is irrevocable.
The goal of Algorithm is to minimize the number of different colors used during the game.

Throughout the paper $\log$ and $\ln$ are logarithm functions to base $2$ and $e$, respectively.
By the \emph{size} of a graph we mean the number of vertices in the graph.

For most classes of graphs the number of colors necessary in the corresponding on-line coloring game can not be bounded in terms of the chromatic number of the constructed graph.
Rare examples of classes where it is possible include interval graphs~\cite{KiersteadT81}, more generally cocomparability graphs~\cite{KiersteadPT94} or $P_5$-free graphs~\cite{KiersteadPT95}.
All of these results are covered by the main result of~\cite{KiersteadPT94} that says that for any tree $T$ with radius $2$, 
the class of graphs that do not contain an induced copy of $T$ can be colored on-line with number of colors being a function of $T$ and chromatic number of presented graph.

Usually, for general enough classes of graphs, the best one can hope for is to bound the number of colors used in an on-line coloring game in terms of the number of rounds (i.e. the size of the constructed graph).

It is a popular exercise to show a strategy for Presenter that constructs forests of size $n$ and forces Algorithm to use at least $\floor{\log n}+1$ colors.
On the other hand, \emph{First-Fit strategy} for Algorithm (that is a strategy that colors each incoming vertex with the least admissible natural number), uses at most $\floor{\log n}+1$ colors on forests of size $n$.

When the game is played on bipartite graphs, Presenter can easily trick out First-Fit strategy and force $\ceil{\frac{n}{2}}$ colors on a bipartite graph of size $n$.
Lov{\'a}sz, Saks and Trotter~\cite{LovaszST89} gave a simple strategy for Algorithm using at most $2\log n$ colors on bipartite graphs of size $n$.
Recently, Bianchi et al.~\cite{BianchiBHK14} proposed a strategy for Presenter that forces Algorithm to use at least $\floor{1.13746\log n - 0.49887}$ colors on bipartite graphs of size $n$.
We improve this bound to $2\log n - 10$, which matches an upper bound from~\cite{LovaszST89} up to an additive constant.

\begin{theorem}\label{thm:bipartite}
There exists a strategy for Presenter that forces at least $2\log_2 n - 10$ colors on bipartite graphs of size $n$.
\end{theorem}

Perhaps, the most exciting open problem in the area is whether there is a strategy for Algorithm using $O(n^{1-\epsi})$ colors on triangle-free graphs of size $n$, for some $\epsi>0$.
The only non-trivial on-line algorithm for triangle-free graphs, given by Lov{\'a}sz et al.~\cite{LovaszST89}, uses $O(\frac{n}{\log\log n})$ colors on graphs of size $n$.
Triangle-free graphs may have arbitrarily large chromatic number, 
but the chromatic number in this case has a precise bound in terms of the number of vertices.
Ajtai, Koml{\'o}s and Szemer{\'e}di~\cite{AjtaiKS80} proved that the chromatic number of any triangle-free graph of size $n$ is $\O{\frac{n}{\log n}^\frac{1}{2}}$.
Kim~\cite{Kim95} presented a probabilistic construction of triangle-free graphs with chromatic number $\Om{\frac{n}{\log n}^\frac{1}{2}}$.

The \emph{girth} of a graph $G$ is the length of the shortest cycle in $G$.
The \emph{odd-girth} of $G$ is the length of the shortest odd cycle in $G$.
Thus, triangle-free graphs have odd-girth at least $5$.
For bipartite graphs the odd-girth is not defined and for convenience it is set to $\infty$.
For graphs with odd-girth at least $7$, i.e.~$C_3$- and $C_5$-free, Kierstead~\cite{Kierstead98} gave a strategy for Algorithm that uses $\O{n^\frac{1}{2}}$ colors on such graphs of size $n$.
Curiously, no better strategy for Algorithm is known even for classes of graphs with odd-girth larger than any $g \geq 7$.
On the off-line side, Denley~\cite{Denley94} has shown that the chromatic number of graphs of size $n$ with odd-girth at least $g\geq7$ is $\O{\frac{n}{\log n}^\frac{2}{g-1}}$.
For the lower bound, it is well-known (see Lemma~6.1 in Krivelevich~\cite{Krivelevich97}) that there are graphs of size $n$ and with girth at least $g$ with chromatic number $\Om{n^\frac{1}{g-2}}$.

For a good introduction to our second result, we present a simple strategy for Presetner by Diwan, Kenkre and Vishwanathan~\cite{DiwanKV13} that forces Algorithm to use $\Om{n^\frac{1}{2}}$ colors on triangle-free graphs of size $n$.
Note that it improves the bound that trivially follows from the non-trivial off-line construction only by a logarithmic factor.
For the case of graphs with odd-girth at least $7$ we propose a strategy for Presenter that forces Algorithm to use $\Om{\frac{n}{\log n}^\frac{1}{3}}$ colors.

\begin{theorem}\label{thm:no-C3-and-C5}
There exists a strategy for Presenter that forces $\Om{\frac{n}{\log n}^\frac{1}{3}}$ colors on graphs of size $n$ with odd-girth at least $7$.
\end{theorem}

\section{Bipartite graphs}\label{sec:bipartite}

\begin{proof}[Proof of Theorem \ref{thm:bipartite}]
We give a strategy for Presenter that forces Algorithm to use $c$ different colors on bipartite graphs of size $\brac{8+7\sqrt{2}}2^{\frac{c}{2}}$.
Thus, Presenter can force $\floor{2\log n - 2\log(8+7\sqrt{2})} \geq \floor{2\log n - 8.32}$ colors on bipartite graphs of size $n$.

At any moment during the game the presented graph is bipartite and consists of a number of connected components.
Each component has the unique bipartition into two independent sets which we call the \emph{sides} of the component.
A color $\alpha$ is \emph{one-sided} in a component $C$ if there is a vertex in $C$ colored with $\alpha$ but only in one out of the two sides of $C$.
A color $\alpha$ is \emph{two-sided} in $C$ if there are vertices in both sides of $C$ colored with $\alpha$.
The set of two-sided colors in a component $C$ is denoted by $\ts(C)$.
The \emph{level} of a component $C$, denoted by $\lev(C)$, is the number of two-sided colors in $C$.

The strategy is divided into phases in which carefully chosen components are merged or a new component being a single edge is introduced.
After each phase, for each presented component $C$ the strategy maintains two \emph{selected} vertices in opposite sides of $C$ colored with one-sided colors.
The two-element set of colors assigned to the selected vertices of a component $C$ is denoted by $\sel(C)$.

Consider a single phase of the strategy.
If a new component $C$, which is always a single edge, is introduced then $C$ has no two-sided colors and both vertices of $C$ are selected.
If the strategy \emph{merges} components $C_1,\ldots,C_k$, all with the same level, then for every $C_i$ the strategy fixes one side of $C_i$ to be the left side, and the other to be the right side.
After that, two adjacent vertices $a$ and $b$ are introduced.
The vertex $a$ is adjacent to all the vertices in the left sides of components $C_1,\ldots,C_k$ and $b$ is adjacent to all the vertices in the right sides.
Let $C$ be the component created by this merge.
The colors assigned to $a$ and $b$ are one-sided colors in $C$ and the strategy chooses $a$ and $b$ to be selected for $C$.
Observe that $\ts(C) = \bigcup_{i=1}^k \ts(C_i) \cup X$, where a color $\alpha$ is in $X$ if there are different $C_i$ and $C_j$ such that $\alpha$ is one-sided in both $C_i$ and $C_j$, $\alpha$ appears on a vertex in the left side of $C_i$ and on a vertex in the right side of $C_j$.

A single phase is described by the following four rules.
For each phase the strategy uses the first applicable rule.
\begin{enumerate}
\item \label{rul:dif} \textbf{Merge Different.}
  If there are two components $C_1$, $C_2$ with the same level and $\norm{\ts(C_1) \setminus \ts(C_2)} \geq 2$,
  then merge those two components into a new component $C$.
  Note that $\lev(C) \geq \lev(C_1) + 2$ and $\norm{C} = \norm{C_1} + \norm{C_2} + 2$.
\item \label{rul:sim} \textbf{Merge Similar.}
  If there are two components $C_1$, $C_2$ with the same level and $\norm{\ts(C_1) \setminus \ts(C_2)} = 1$ and $\sel(C_1) \cap \sel(C_2) \neq \emptyset$,
  then merge those two components into a new component $C$ in such a way that a common one-sided color becomes two-sided in $C$.
  Note that $\lev(C) \geq \lev(C_1) + 2$ and $\norm{C} = \norm{C_1} + \norm{C_2} + 2$.
\item \label{rul:cyc} \textbf{Merge Equal.}
  If there are $k \geq 2$ components $C_1, \ldots, C_k$ with the same level and $\ts(C_1) = \ldots = \ts(C_k)$
  and there are $k$ distinct colors $\alpha_1, \ldots, \alpha_k$ such that
  $\sel(C_1) = \set{\alpha_1, \alpha_2}$, $\sel(C_2) = \set{\alpha_2, \alpha_3}$, $\ldots$, $\sel(C_k) = \set{\alpha_k, \alpha_1}$,
  then merge those $k$ components into a new component $C$ in such a way that
  each of the $\alpha_i$'s becomes two-sided in $C$.
  Note that $\lev(C) \geq \lev(C_1) + k$ and $\norm{C} = \sum_{i=1}^k \norm{C_i} + 2$.
\item \label{rul:int} \textbf{Introduce.} Introduce a new component $C$ being a single edge.
  Note that $\lev(C) = 0$ and $\norm{C} = 2$.
\end{enumerate}
This concludes the description of a single phase of the strategy.
Now we present two simple invariants kept by the strategy.

\textbf{Invariant 1.} $\norm{C} \leq 2^{\frac{\lev(C)}{2}+2}-2$ for every component $C$ after each phase of the strategy.

The statement vacuously holds at the beginning of the on-line game.
We show that the invariant holds from phase to phase.
Clearly, it suffices to argue that a component $C$ created in a considered phase satisfies the statement.
If $C$ is a product of $k\geq2$ components $C_1,\ldots,C_k$ by one of the merging rules~\ref{rul:dif}-\ref{rul:cyc} then
\begin{align*}
\norm{C}&=\sum_{i=1}^k\norm{C_i}+2\leq \sum_{i=1}^k\left( 2^{\frac{\lev(C_i)}{2}+2}-2\right) +2\\
&\leq k\cdot2^{\frac{\lev(C)-k}{2}+2}-2(k-1) = 2^{\frac{\lev(C)}{2} -\frac{k}{2} + \log k +2}-2(k-1)\\
&\leq 2^{\frac{\lev(C)}{2}+2}-2.
\end{align*}
If $C$ is a component introduced by Rule~\ref{rul:int}, then invariant holds trivially.

\textbf{Invariant 2.} After each phase, none of the merging rules~\ref{rul:dif}-\ref{rul:cyc} applies to the set of all available components without the component created in the last phase.
\TODO{Previous sentence is so bad that I want to cry!}

The statement trivially holds after the first phase.
Let $C_k$ be the component created in the $k$-th phase.
If $C_k$ is introduced by Rule~\ref{rul:int}, then rules~\ref{rul:dif}-\ref{rul:cyc} do not apply in the graph without $C_k$.
If $C_k$ is merged from some components available after phase $k-1$, then by induction hypothesis $C_{k-1}$ is one of the merged components.
Thus, the set of components after phase $k$ without $C_k$ is a subset of the set of components after phase $k-1$ without $C_{k-1}$.
Therefore, the statement follows by induction hypothesis.

For the further analysis we need one more observation.
Suppose that Rule~\ref{rul:cyc} does not apply to the current set of components and there is a set of colors $T$ and $p$ distinct components $C_1,\ldots,C_p$ with $\ts(C_i)=T$ for all $i\in\set{1,\ldots,p}$.
We claim that $\norm{\bigcup_{i=1}^p\sel(C_i)}\geq p+1$.
Indeed, consider a multigraph $M$ with vertex set $\bigcup_{i=1}^p\sel(C_i)$ and $p$ edges formed by all pairs of colors on selected vertices in the components (i.e. for each $C_i$ there is an edge connecting colors in $\sel(C_i)$).
If $\norm{\bigcup_{i=1}^p \sel(C_i)}\leq p$ then there is a cycle in $M$.
Say that a cycle is defined by $q\geq2$ edges originating from components $C_{i_1},\ldots,C_{i_q}$.
Then Rule~\ref{rul:cyc} can be applied to components $C_{i_1},\ldots,C_{i_q}$, contradicting our assumption.

Fix $c\geq3$ and  consider the situation in a game after a number of phases.
Suppose that Algorithm used so far fewer than $c$ colors.
We are going to argue that Presenter introduced fewer than $\brac{8+7\sqrt{2}}2^\frac{c}{2}$ vertices and this will conclude the proof.
Clearly, for any available component $C$ we have $\lev(C)\leq c-3$ and therefore by Invariant 1 we have $\norm{C} \leq 2^{\frac{c-3}{2}+2}-2$.
Let $\calC$ be the set of all components but the one created in the last phase.
By Invariant 2 none of the merging rules~\ref{rul:dif}-\ref{rul:cyc} applies to $\calC$.

Fix $\ell\in\set{0,\ldots,c-3}$ and let $\calC_\ell$ be the set of $C\in \calC$ with $\lev(C)=\ell$.
Now, we want to bound the number of components in $\calC_\ell$.
Let $\{T_1,\ldots,T_m\}$ be the set of values of $\ts(C)$ attained for $C\in\calC_\ell$.
Let $p_i$ be the number of components $C\in\calC_\ell$ with $\ts(C)=T_i$ and let $S_i=\bigcup_{\substack{C\in\calC_\ell,\\\ts(C)=T_i}}\sel(C)$.
As Rule~\ref{rul:dif} can not be applied $\norm{T_i-T_j}=1$ for all distinct $i,j\in\set{1,\ldots,m}$.
In particular, $\norm{T_1\cap S_i}\leq 1$ for all $i\in\set{2,\ldots,m}$.
As Rule~\ref{rul:sim} can not be applied $S_i\cap S_j=\emptyset$ for all distinct $i,j\in\set{1,\ldots,m}$.
As Rule~\ref{rul:cyc} can not be applied $\norm{S_i}\geq p_i+1$ for all $i\in\set{1,\ldots,m}$.
All this give a lower bound on the number of colors in $\bigcup_{i=1}^m S_i \cup T_1$.
\begin{align*}
c-1\geq \norm{\bigcup_{i=1}^m S_i \cup T_1} &= \sum_{i=1}^m \norm{S_i} + \norm{T_1} - \sum_{i=1}^m\norm{T_1 - S_i}\\
&\geq \sum_{i=1}^m(p_i+1)+\ell-(m-1).
\end{align*}
Thus, $\norm{\calC_\ell}=\sum_{i=1}^m p_i \leq c-\ell-2$.

We define an auxiliary function $f(i)=\sum_{j=0}^n (i-j)2^{\frac{j}{2}}$ and observe an easy bound 
$f(i) < \brac{4+3 \sqrt{2}} 2^{\frac{i}{2}}$.
Using Invariant~1 and the bound on the number of components with any particular level, we get the following bound on the total number of vertices within components in $\calC$:
\begin{align*}
 \sum_{\ell=0}^{c-3}\norm{\calC_\ell}\cdot\brac{2^{\frac{\ell}{2}+2}-2}&\leq \sum_{\ell=0}^{c-3}(c-\ell-2)\cdot\brac{2^{\frac{\ell}{2}+2}-2}\\
  &< \sum_{\ell=0}^{c-2}(c-2-\ell)\cdot4\cdot2^{\frac{\ell}{2}} = 4f(c-2)\textrm{.}
\end{align*}


To finish the proof we sum up the upper bounds on the size of the last produced component and the total size of all remaining components, that is
\[
  2^{\frac{c-3}{2}+2}-2 + 4\brac{4+3\sqrt{2}} 2^{\frac{c-2}{2}} < \brac{8+7\sqrt{2}}2^\frac{c}{2}\textrm{.}
\]
\end{proof}

\section{The odd-girth}\label{sec:oddgirth}
In this section we consider classes of graphs with odd-girth bounded from below.
The high value of odd-girth of a graph implies that the graph is locally bipartite.
However, it seems hard to exploit this property in an on-line framework.
The only known on-line algorithms, that can use large odd-girth, 
is the one by Lov\'{a}sz et al.~\cite{LovaszST89} using $O(\frac{n}{\log\log n})$ colors for graphs of girth at least $4$ and of size $n$,
and the one by Kierstead~\cite{Kierstead98} using $\O{n^\frac{1}{2}}$ colors on graphs of size $n$ with odd-girth at least $7$.
We present constructions that prove lower bounds for these problems for odd-girth at least $5$ and at least $7$.

\begin{theorem}[Diwan, Kenkre, Vishwanathan~\cite{DiwanKV13}]\label{thm:triangle-free}
There exists a strategy for Presenter that forces $\Om{n^\frac{1}{2}}$ colors on triangle-free graphs of size $n$.
\end{theorem}

\begin{proof}
We give a strategy for Presenter that forces Algorithm to use $c$ different colors.
An auxiliary structure used by Presenter during the game is a table with $c$ rows and $c$ columns.
Each cell of the table is initially empty, but over the time Presenter puts vertices into the table.
There will be at most one vertex in each cell.
Each time Algorithm colors a vertex with color $i$, the vertex is put into the last empty cell in the $i$-th row,
e.g.\ the first vertex colored with $i$ ends up in the cell in the $i$-th row and $c$-th column.
Anytime Algorithm uses $c$ different colors Presenter succeeds and the construction is complete.

The strategy is divided into $c$ phases numbered from $0$ to $c-1$.
Let $I_0=\emptyset$ and for $k>0$ let $I_k$ be the set of vertices in the $k$-th column at the beginning of phase $k$.
Now, as long as there is no vertex in the $(k+1)$-th column with a color different than all the colors of the vertices in $I_k$,
Presenter introduces new vertices adjacent to all vertices in $I_k$ (and no other).
The phase ends when some vertex $v$ is put in the $(k+1)$-th column.
Clearly, each phase will end as Algorithm must use colors different than colors used on $I_k$ for all vertices presented in the $k$-th phase and the space in the table where these vertices are stored is limited.
Observe also, that the vertices presented in the $k$-th phase are never put into the $k$-th column.
Hence, each edge connects vertices in different columns and each column in the table forms an independent set.
As Presenter introduces only vertices with neighborhood contained within one column the constructed graph is triangle-free.
Observe also that $I_{k+1}$ is strictly larger than $I_{k}$ and therefore $\norm{I_{k}}\geq k$.
Thus, after the completion of the $(c-1)$-th phase each cell of the $c$-th column is filled and therefore Algorithm has already used $c$ colors.
As there are only $c^2$ cells in the table, Presenter introduced at most that many vertices.
\end{proof}

Note that when Algorithm uses First-Fit strategy, the graphs constructed by Presenter are exactly shift graphs (a well known class of triangle-free graphs with chromatic number logarithmic in terms of their size).
In the same vein the next strategy, for graphs with odd-girth at least $7$, is inspired by the construction of double-shift graphs.

\begin{proof}[Proof of Theorem~\ref{thm:no-C3-and-C5}]
We describe a strategy for Presenter that forces Algorithm to use $c$ different colors.
This time an auxiliary structure used by Presenter is a table with $c$ rows and $3c$ columns.
Each cell of the table is initially empty, but over the time Presenter puts vertices into the table.
There will be at most $3c$ vertices in each cell.
All vertices put in the $i$-th row will be of color $i$.
Every vertex in the table will have assigned a non-negative weight.
If at some point Algorithm uses $c$ different colors, Presenter succeeds and the construction is complete.

Color $i$ is \emph{available} for the $j$-th column,
 or the cell in the $i$-th row and $j$-th column is available,
when there are fewer than $3c$ vertices in this cell.
Otherwise, the color, or the cell is \emph{blocked}.

The strategy is divided into $3c$ phases numbered from $0$ to $3c-1$.
In the $k$-th phase, Presenter selects $3c$ groups of vertices that are already in the table.
If $k=0$ or there are no blocked cells in the $k$-th column then all the groups are empty.
Otherwise, Presenter splits vertices of blocked colors for the $k$-th column in such a way that each group contains a vertex of each blocked color and no other vertices.
Now for each group $R$, Presenter plays according to the following rules.

\Rule1 If there is a color available for the $k$-th column but blocked for all the columns to the right of the $k$-th column (i.e. for columns $k+1, \ldots, 3c-1$), then the whole phase is finished and Presenter starts the next phase.
Phases that ended for this reason are called \emph{broken}.
Clearly, the number of broken phases is bounded by the number of colors, that is by $c$.

Otherwise, Presenter introduces an independent set $F$ of $3c\brac{1+\ceil{\ln3c}}$ new vertices adjacent to all the vertices in $R$.
Set $F$ is called a \emph{fan}.
Now, Presenter investigates the possibilities of putting some of the vertices in $F$ into the table but he restricts himself to put them only into one column which is to the right of the $k$-th column.
We are going to use this property in the proof that the constructed graph contains neither $C_3$ nor $C_5$ as a subgraph.

\Rule2 If there is a cell in the table in a column to the right of the $k$-th column, say the cell in the $i$-th row and $j$-th column, such that there are $m<3c$ vertices in the cell (in particular the cell is available) and there are at least $3c-m$ vertices in $F$ colored by Algorithm with $i$, then Presenter puts $3c-m$ vertices in $F$ colored with $i$ into the cell.
From now on this cell is blocked.
All vertices put into the table by this rule receive weight $0$.
All the other vertices in $F$ (with color different than $i$) are \emph{discarded} and will not be used as neighbors for vertices introduced in the future.

\Rule3 If there is no such cell (i.e. we cannot apply Rule 2), then we call $F$ an \emph{interesting fan}.
Let $t_i$, for $i\in\set{1,\ldots,c}$ be the number of columns to the right of the $k$-th column for which color $i$ is available.
Consider a bipartite graph with one part formed by $3c-1-k$ columns to the right of the $k$-th column and the second part formed by $3c\brac{1+\ceil{\ln3c}}$ vertices in $F$.
We put an edge in the graph between the $j$-th column and a vertex $v\in F$ of color $i$ if color $i$ is available for the $j$-th column.
To vertex $v\in F$ of color $i$ we assigned weight $\frac{3c}{t_i}$.
All its incident edges also get weight $\frac{3c}{t_i}$.
The total weight of all the edges in the graph is
\[
\sum_{i\in\set{1,\ldots,c}} \sum_{\substack{v\in F\\\textrm{$v$ is colored with $i$}}} t_i\cdot\frac{3c}{t_i}=\norm{F}\cdot 3c.
\]
This means that for some column $j>k$, the total weight of incident edges is at least $\norm{F}\cdot \frac{3c}{3c-1-k}\geq \norm{F}$.
Presenter puts all the vertices in $F$ colored with available colors for the $j$-th column into the $j$-th column.
The remaining vertices in $F$ are discarded.
Note that after this step all the cells in the $j$-th column still have at most $3c$ vertices as otherwise Presenter would use Rule 2.
Note also that the total weight assigned to the vertices put into the table is at least $\norm{F}$.
This finishes the description of the strategy for Presenter.

Note that there are at most $3c$ phases of the game, and that in every phase at most $3c$ fans of size $3c\brac{1+\ceil{\ln3c}}$ each, are presented.
This gives no more than $27 c^3 \brac{1+\ceil{\ln3c}}$ vertices in total.
We claim that before the end of the last phase Algorithm has to use $c$ different colors.
Suppose it does not.
At most $c$ of the phases are broken.
During remaining phases Presenter introduces at least $2c\cdot3c$ fans.
As there are $c\cdot3c$ cells in the table, and each cell can be blocked only once, at most $3 c^2$ fans are used to block some cell.
Thus, we have at least $3c^2$ interesting fans produced in the construction.
The total number of vertices in the interesting fans is at least $3c^2\cdot 3c\brac{1+\ceil{\ln3c}}$.
Recall that the total weight of vertices put into the table during the construction upper bounds the total size of all the interesting fans.
We conclude that the weight of all the vertices put into the table during the construction is at least $9c^3\brac{1+\ceil{\ln3c}}$.
Observe that for each row $i$ of the table and for each $t=1,\ldots,3c$ the number of vertices with weight greater or equal to $\frac{3c}{t}$ is at most $3c\cdot t$.
Indeed, let $v$ be the first vertex put into the $i$-th row with weight at least $\frac{3c}{t}$.
This means that there are at most $t$ available cells in the $i$-th row at the time when $v$ is put into the table.
Since in each cell there are at most $3c$ vertices we know that at most $3c\cdot t$ more vertices may end up in the $i$-th row.
Thus, there are at most $3c$ vertices of weight $\frac{3c}{1}$ and at most $6c$ vertices of weight at least $\frac{3c}{2}$ and so on, and it is easy to see that the total weight of the vertices in any row is at most
\[
  \sum_{t=1}^{3c} \frac{3c}{t}\cdot 3c = 9c^2\sum_{t=1}^{3c}\frac{1}{t} < 9c^2 \brac{1+\ln3c}\textrm{.}
\]

If the game does not end before phase $3c$ then vertices of total weight at least $9c^3\brac{1+\ceil{\ln3c}}$ are put into the table.
However, total weight of vertices in any row is strictly smaller than $9c^2 \brac{1+\ln3c}$.
Since there are $c$ rows, we get a contradiction.

During the game Presenter introduced at most $27c^3\brac{1+\ceil{\ln3c}}$ vertices and forced Algorithm to use $c$ colors.
Inverting the function, we get that Presenter forces Algorithm to use $\Om{\frac{n}{\log n}^{\frac{1}{3}}}$ colors on graphs of size $n$.

To finish the proof, we need to argue that the constructed graph contains neither $C_3$ nor $C_5$ as a subgraph.
For $C_3$ observe that any vertex $v$ introduced by Presenter has all the neighbors contained within one column (at the moment of introduction).
Moreover, the vertex $v$ itself ends up in the column to the right of the column of its neighbors or is discarded.
This implies that vertices in a single column form an independent set.
Since the neighborhood of each vertex at the moment of introduction is an independent set the whole graph is triangle-free.

Now, assume to the contrary that $C_5$ is contained in the constructed graph.
As the graph does not contain $C_3$, the copy of $C_5$ is an induced subgraph on some vertices $v_0,v_1,v_2,v_3,v_4$.
We can assume, that all these vertices are put into the table.
If it is otherwise, then we can exchange a discarded vertex $u$ to any of the non discarded vertices, say $w$, in the same fan.
All neighbors of $u$ are also neighbors of $w$ and we still get a $C_5$.
We can choose $v_2$ to be the vertex in the left most column of the five possibilities.
By the construction, we know that $v_1$ and $v_3$ are introduced in a single fan of vertices adjacent to the group of vertices containing $v_2$.
This implies that $v_1$ and $v_3$ are in the same column.
All the neighbors of $v_1$ in the columns to the left of $v_1$ are in the same group (and column) as $v_2$.
The same holds for $v_3$.
If $v_0$ and $v_4$ were both in columns to the left of the column of $v_1$ and $v_3$ then they are in the same column as $v_2$ and $v_0$ is not adjacent to $v_4$.
So at least one of $v_0$ or $v_4$ is in a column to the right of $v_1$.
Say it is $v_0$.
If $v_4$ is in a column to the left of $v_0$ then it is in the same column as both $v_1$ and $v_3$ and $v_4$ is not adjacent to $v_3$.
If $v_4$ is in a column to the right of $v_0$ then both $v_0$ and $v_3$ are in different columns to the left of $v_4$, and one of them is not adjacent to $v_4$.
\end{proof}


\bibliographystyle{plain}
\bibliography{paper}
\end{document}